\documentclass[11pt]{article}
\usepackage{amsfonts}
\newcommand{\beq}{\begin{equation}}
\newcommand{\enq}{\end{equation}}

\newtheorem{Theorem}{Theorem}[section]
\newtheorem{Lemma}[Theorem]{Lemma}
\newtheorem{Definition}[Theorem]{Definition}
\newtheorem{Remark}[Theorem]{Remark}

\newcommand{\benu}{\begin{enumerate}}
\newcommand{\beqa}{\begin{eqnarray}}
\newcommand{\beqan}{\begin{eqnarray*}}
\newcommand{\eay}{\end{array}}
\newcommand{\edm}{\end{displaymath}}
\newcommand{\eenu}{\end{enumerate}}
\newcommand{\eeq}{\end{equation}}
\newcommand{\eeqa}{\end{eqnarray}}
\newcommand{\eeqan}{\end{eqnarray*}}

\newcommand{\br}{\begin{Remark}}
\newcommand{\er}{\end{Remark}}

\newcommand{\bqa}{\begin{eqnarray}}
\newcommand{\eqa}{\end{eqnarray}}
\newcommand{\bqw}{\begin{eqnarray*}}
\newcommand{\eqw}{\end{eqnarray*}}

\newcommand{\bea}{\begin{array}{cc}}
\newcommand{\ena}{\end{array}}

\RequirePackage{amsfonts}
\usepackage{stmaryrd}
\usepackage{amssymb}
\usepackage{mathrsfs}
\usepackage{amsmath}
\usepackage{subeqnarray}
\usepackage{verbatim}
\usepackage{cases}
\usepackage{amsfonts}
\usepackage{amssymb}
\usepackage{amssymb}
\usepackage{amssymb}
\usepackage{hyperref}
\usepackage[utf8]{inputenc}
\begin{document}

\newpage
\pagenumbering{arabic} \setcounter{page}{1}

\begin{center}
	{\Large  Forced oscillation of a damped BBM equation posed on whole line in low regularity spaces}\\\vspace{0.25in} 
    Chun-Ho Lau \footnote[2]{Department of Mathematics, National Taiwan Normal University, Taipei City, Taiwan 106308. Email: lauchho@ntnu.edu.tw}
	Taige Wang \footnote[1]{Department of Mathematical Sciences, University of Cincinnati, Cincinnati, OH 45221, USA. Corresponding to: taige.wang@uc.edu}
\ \ \ \
	   \vspace{0.06in}
\end{center}
\vspace{0.06in}
\begin{abstract}
	In this manuscript, we would established in low regularity  spaces $H^\ell, \ell\in [0,1)$, the existence and stability results of time-periodic solution of 1D Cauchy problem of forced damped Benjamin-Bona-Mahony equation (BBM). We use estimates from I-energy method to derive needed estimates in $H^\ell$ for the linearized problem, then convection term will be treated as perturbation of linear problem such that original Cauchy problem is solved. 
	
\end{abstract}

\section{Introduction}\label{sec1}
\setcounter{equation}{0}

 In this manuscript, we would study the existence and stability of general temporal-periodic solutions to Cauchy problem of a weakly damped Benjamin-Bona-Mahony equation (regularized long-wave equation or BBM equation) posed on whole line:

\begin{equation}\label{1-1}
u_t-u_{xxt} + u - u_{xx} + uu_x = f,\quad  (x, t)\in\mathbb{R}\times \mathbb{R}_+,
\end{equation} prescribed with initial data $u(x, 0)=\phi(x)$. The well-posedness theory such as local well-posedness (LWP) and global well-posedness (GWP) of this type of equation corresponding to particularly shaped initial data on whole line and torus  were discussed by J. Bona and H. Chen in \cite{BC, Chen}. Added to BBM equation, terms $-u_{xx}$ and/or zero-order damped term $u$ appear on the left side of the equation and play key roles in regularizing and stabilizing. The external force $f$ is applied at each point of fluid media and must be input with small amplitude, hence it will be $f(x, t)\in\mathbb{R}\times \mathbb{R}_+$ through the entire paper, and in our experiment setting, we would like this force being temporal-periodic, thus solution $u$ generated turns out to be periodic in time as one of the results. The damping is quite important determining the stability of solution. In PDEs such as KdV (Korteweg-de Vries) and parabolic equations, the existent damping guarantees a stabilized analytic semigroup (\cite{BSZ2, WZ, GWX} follow this mechanism to obtain results for unbounded and bounded domains, respectively); J. Bona and J. Wu also analyzed the particular damping effects of adding regularizing terms such as $-u_{xx}$ or $u$ in \cite{BW} for KdV. For BBM equation on whole line, similar damping mechanism to stabilizing was obtained (see e.g. \cite{Amick} by C. Amick, J. Bona and M. Shonbek).\\

 We would pursue those properties of the general solutions of BBM equation in $H^\ell, \ell\in [0, 1),$ given $\phi\in H^\ell, f\in L^2$. Such framework of solution theory of KdV equations could be even extended to be posed in negative indexed Sobolev spaces or Bourgain spaces were inspired by seminal work quantifying this methodology by J. Bourgain \cite{Bourgain}, and later implemented by C. Kenig, G. Ponce, and L. Vega (see e.g. \cite{Kenig1, Kenig2}) and J. Bona and B.-Y. Zhang et al (see e.g. \cite{BZ, YZ}). In this article, generality of solutions is referred to those in below-energy or low-regularity space with positive indices ($H^\ell$ with $\ell\in [0, 1)$ which is below $H^1$ and rough). Regularity can not be lower in negative indices due to the finding of sharp regularity in \cite{Tzvet} by J. Bona and N. Tzvetkov. Solution theory in $H^1$ has been established for decades since mathematical theory were discussed along with KdV and BBM equations in \cite{BB, BBM} in 1970s. We will use the particular method: \emph{I-energy method} technically to construct an $H^\ell$-equivalent norm defined via operator $I_N: H^{\ell}\mapsto H^1$, for an intentionally selected $N$; this technique has been founded on low-regularities of solutions to KdVs by J. Colliander et al \cite{Tao1, Tao2}. It is worthy to point out that $I_Nu$ then satisfies the PDE which can be approached by classical energy method used in \cite{BSZ2, BSZ3, Zhang} by J. Bona, S. Sun, and B.-Y. Zhang, and \cite{GWX, LW} by current authors recently. \\ 

Seen from works \cite{BSZ2, BSZ3, Zhang}, the linear problem of KdV equation can be solved by formulating mild solution, then bilinear estimates developed leads to the estimate of convection term $uu_x$ thus the solvability of nonlinear problem. In current manuscript, this methodology brings results in function spaces $H^\ell$ and/or $Y^\ell_{\tau, T}$ . In light of the application facilitating existence theory of mild solutions, B.-Y. Zhang et al including one of the authors considered the time-periodic solutions to KdV equations and dissipative PDEs, respectively in 1D finite segments and 2D finite domains (see e.g. \cite{UsmanZ1, UsmanZ2, Zhang, WZ, GWX, CWX}). On the dynamical system behavior of damped KdV, existence of global attractors of solutions in torus $\mathbb{T}$ has been discussed in \cite{Goube1, Goube2} by O. Goubet et al and \cite{Tsugawa, Yang} in negative indexed Sobolev spaces, respectively. By means of I-energy method and high-low frequency decomposition, GWP of BBM equation was obtained by J. Bona and N. Tzvetkov in \cite{Tzvet}. Global attractors of damped BBM equation but in positively indexed Sobolev spaces $H^\ell, \ell\in[0, 1)$ was discussed by M. Wang \cite{MWang, MWang2}. In light of the argument from \cite{Tzvet}), the above dynamical behavior was obtained for  $\ell\ge 0$.  Following \cite{Tzvet} and \cite{MWang, MWang2}, the time-periodic solution in our paper is discussed for $\ell\in[0, 1)$.\\

 The entire paper will be organized as follows: we would present notations, related definitions, and main results in Section 2, and preliminary lemmas in Section 3. Proof of theorems will be in Section 4. 

\section{Main results}\label{sec2} 
\setcounter{equation}{0}

Notation $\|\cdot\|_X$ denotes the endowed canonical norm of a specific Banach space $X$, and $\|\cdot\|$ denotes norm of square-integrable space $L^2(\mathbb{R})$. $\widehat \cdot$ represents Fourier transform $\mathcal{F}\cdot$ defined in frequency space as a function of $\xi\in\mathbb{R}$. $(\cdot, \cdot)$ denotes the classical inner product in $L^2(\mathbb{R})$. In the context of flow dynamics, we would like to use ODE fashion such as $u(t)$ or $\|u(t)\|_{X}$ to present the evolution along increasing time. \\

 We will construct solution related to function space close to $H^1(\mathbb{R})$ on $t\in[\tau, \tau+T]$: 

 \begin{equation*}
Y^1_{\tau, T} = \left\{u\big|\sup_{\tau\le s\le \tau+T}(\|u(s)\|+\|u_x(s)\|)+\left(\int_{\tau}^{t+T}(\|u(s)\|^2+\|u_x(s)\|^2)ds\right)^{1\over 2}< \infty\right\}
 \end{equation*} with norm
 \begin{equation*}
\|u\|_{Y^1_{\tau, T}}=\sup_{\tau\le s\le \tau+T}(\|u(s)\|+\|u_x(s)\|)+\left(\int_{\tau}^{\tau+T}(\|u(s)\|^2+\|u_x(s)\|^2)ds\right)^{1\over 2}.
 \end{equation*}

 Likewise, we also can define function space $Y^\ell_{\tau, T}$ related to $H^\ell(\mathbb{R})$ by norm

  \begin{equation*}
\|u\|_{Y^\ell_{\tau, T}}=\sup_{\tau\le s\le \tau +T}\|u(s)\|_{H^\ell}+\left(\int_{\tau}^{\tau+T}\|u(s)\|^2_{H^\ell}ds\right)^{1\over 2}
 \end{equation*} with $\ell\in [0, 1)$. \\

 Before we state our results in this manuscript, we would distinguish two types of stability of solution in the following definitions:

 \begin{Definition}\label{de2-2}
    Regarding to dynamics of differential equations, we say solution $\tilde u(t)$ generated by any initial data $u_0\in X$ is locally stable, if $u(t)$ converges to $\tilde u$ in $X$ as $t\rightarrow \infty$, when $u_0$ is sufficiently close to $\tilde u$.
\end{Definition}


\begin{Definition}\label{de2-3}
    We say $\tilde u(t)$ is globally stable, if  $u(t)$ generated by any initial data $u_0\in X$ converges to $\tilde u$ in $X$ no matter how far $u_0$ is from $\tilde u$.
\end{Definition}

 Our main results for weekly damped BBM equation are presented right now. First, there is local-well-posedness: 

\begin{Theorem}\label{th2-1}
	Given $\tau\ge0, T>0, \phi\in H^\ell, f\in L^2$ for $\ell\in[0, 1)$, if there exists a constant $\delta>0$ and holds smallness condition $\|\phi\|_{H^{\ell}}+ \sup_{t\ge 0}\|f(t)\|\le \delta$, there holds local well-posedness of solution $u$ to Cauchy problem (\ref{1-1}) where $u$ satisfies
	
	\begin{equation}\label{2-1}
	\|u\|_{Y^{\ell}_{\tau, T}}\le C(\|\phi\|_{H^\ell}+\sup_{t\ge 0}\|f(t)\|).
	\end{equation} 
\end{Theorem}

Thanks to local well-posedness, we can obtain the temporal-periodic solution $\tilde u(x,t)$ and its stability of solution according to proposed Definitions \ref{de2-2} and {\ref{de2-3}} summarized in the following theorem: 

\begin{Theorem}\label{th2-2}
If conditions in Theorem \ref{th2-1} hold plus $f$ is time-periodic with period $\theta$, there exists an initial condition $\tilde\phi\in H^{\ell}$ generating unique periodic solution $\tilde u\in H^{\ell}$ with period $\theta$. \\

Moreover, if $\delta$ is sufficiently small, there exist $N_0\in\mathbb{N}$ and $\gamma_0(N_0)>0$ such that the periodic solution possesses local stability with an exponential decay rate $2\gamma_0>0$.\\

Essentially, for any size of initial value $\phi$, if $\sup_{t\in[0, \infty)}\|f(t)\|<\delta,$ one can find $N_1\in\mathbb{N}$ and $\gamma_1(N_1)>0$, it holds the global stability of solution with positive decay rate $2\gamma_1>0$.
\end{Theorem}

\section{Preliminaries}\label{sec3}
\setcounter{equation}{0}

We aim to fit the I-method result to obtain solution theory in $H^\ell, \ell\in[0, 1)$. I-method would be called by defining the I operator. Given an $N>0$, let smooth function $m_N$ in $C^{\infty}(\mathbb{R})$ be

\begin{eqnarray*}
m_N(\xi) = 
\begin{cases}
1, \,\,|\xi|\le N,\\\\
{|\xi|^{s-1}\over N^{s-1}}, \,\, |\xi|\ge 2N. 
\end{cases}
\end{eqnarray*}

Define operator $I_N$ for fixed $N>0$: $H^\ell\mapsto H^1$ such that $\widehat{I_N g}(\xi) = m_N(\xi)\widehat g(\xi)$. \\

There holds a key lemma which discloses the link between function spaces of $I_Ng$ and $g$.

\begin{Lemma}\label{le3-1}
It holds for an $\ell\in [0, 1)$ and $g\in H^\ell$ that 

$$\|I_Ng\|_{H^1}\sim \|g\|_{H^\ell}. $$

In particular, $C\|g\|_{H^\ell}\le \|I_Ng\|_{H^1}\le C{N^{1-\ell}} \|g\|_{H^\ell}$, where $C$ doesn't reply on $N$. 
\end{Lemma}

The lemma can be found in many previous seminal papers applying I-method originally in \cite{Tao1, Tao2} by Colliander et al on global solutions of KdV equations, and followed by \cite{Tzvet} by J. Bona and N. Tzvetkov and then \cite{MWang, MWang2} by M. Wang on damped BBM equation on whole line and torus $\mathbb{T}$. \\

Equation (\ref{1-1}) is split into two problems, that is $u = v+z$ respectively governed by the following two systems: 

\begin{eqnarray}\label{3-1}
\begin{cases}
v_t-v_{txx} + v - v_{xx} = f,\\
v(x,0) = \phi(x),
\end{cases}
\end{eqnarray}

and

\begin{eqnarray}\label{3-2}
\begin{cases}
z_t-z_{txx} + z - z_{xx} + zz_x + (vz)_x = -vv_x,\\
z(x,0) = 0.
\end{cases}
\end{eqnarray}

\begin{Lemma}\label{le3-2}
For linear Cauchy problem (\ref{3-1}), given $\tau\ge 0, T> 0, \phi\in H^\ell$, and $f\in L^2$, then $v\in Y^\ell_{\tau, T}$ and satisfies
\begin{eqnarray*}
\sup_{t\in{[\tau, \tau+T]}}\|v(t)\|_{H^\ell}\le \|\phi\|_{H^\ell} + C\sup_{t\in{[\tau, \tau+T]}}\|f(t)\|, 
\end{eqnarray*} and 
\begin{eqnarray*}
\|v\|_{Y^\ell_{\tau, T}}\le C\|\phi\|_{H^\ell} + C\sup_{t\in [\tau, \tau+T]}\|f(t)\|. 
\end{eqnarray*}
\end{Lemma}
{\bf Proof.} Fixing any $N>0$ and applying $I_N$ on both sides of (\ref{3-1}), one obtains

\begin{eqnarray*}
\partial_t I_Nv - \partial_t I_Nv_{xx} + I_Nv - I_Nv_{xx} = I_Nf.
\end{eqnarray*}

Conduct energy estimate by multiplying $I_Nv$ and take integral on $\mathbb{R}$, then

\begin{eqnarray*}
{d\over dt}(\|I_Nv\|^2+\|\partial_xI_Nv\|^2) + 2(\|I_Nv\|^2+\|\partial_xI_N v\|^2) = (I_Nf, I_Nv),
\end{eqnarray*} where we have used fact that $(\partial_{xx}I_Nu, I_Nu) = (\widehat{\partial_{xx}I_Nu}, \widehat{I_Nu})=(I_N\partial_{xx}u, I_Nu)$. \\

Integration factor method of solving ODE leads to

\begin{eqnarray*}
\|I_N v(t)\|_{H^1}^2 = e^{-2t}\|I_N\phi\|^2_{H^1} + \int_{\tau}^{t}e^{-2(t-s)}\|I_Nf(s)\|^2ds.
\end{eqnarray*}

By Lemma \ref{le3-1}, we reach the $\sup_{t\in[\tau, \tau+T]}\|v(t)\|_{H^\ell}$. Note that 

\begin{eqnarray*}
2\int_0^T\|I_Nv(s)\|^2_{H^1}ds = \|I_Nv(0)\|^2_{H^1} - \|I_Nv(T)\|^2_{H^1} + 2\int_0^T (I_Nf(s), I_Nv(s))ds.
\end{eqnarray*}

The previous estimate and Cauchy-Schwartz inequality lead to the estimate of $\|I_Nv\|_{Y^1_{0, T}}$ thus that of $\|v\|_{Y^\ell_{0, T}}$. 
\hspace*{\fill}$\Box$\\

To handle the nonlinear problem (\ref{3-2}), we need to the following bilinear estimate to treat nonlinear term as perturbation of linear problem: 

\begin{Lemma}\label{le3-3} Given $T>0$, there holds bilinear estimate
\begin{eqnarray}
\int_{0}^T\|(I-\partial_{xx})^{-1}I_N(uv_x)(s)\|_{H^1}ds\le C_T \|I_Nu\|_{Y_{0, T}}\|I_Nv\|_{Y_{0, T}}
\end{eqnarray} where the constant $C_T\sim T$, independent of $N$. 
\end{Lemma}

{\bf Proof.} A related lemma can be found in \cite{Tzvet, MWang} in which

$$\|(I-\partial_{xx})^{-1}I_N(uv)_x\|_{H^1}\lesssim \|I_Nu\|_{H^1}\|I_Nv\|_{H^1}. $$

Via Lemma \ref{le3-1} and taking integration over $[0, T]$, we obtain the desired result. 
\hspace*{\fill}$\Box$\\

\begin{Lemma}\label{le3-4} There holds trilinear estimate for inner product involving $u, v, w\in H^{\ell}$: 

\begin{equation}
|(I_N(uv)_x, I_Nw)|\le CN^{-{3\over 2}+}\|I_Nu\|_{H^1}\|I_Nv\|_{H^1}\|I_Nw\|_{H^1},
\end{equation} where $C$ doesn't rely on $N$.
\end{Lemma}
{\bf Proof.} This is a version of such trilinear estimates obtained in Lemma 2.6 in M. Wang \cite{MWang} and \cite{Tzvet}, whose proof is analogous to result therein

$$|(I_N(w^2)_x, I_Nw)|\lesssim N^{-{3\over 2}+}\|I_Nw\|^3_{H^1}. $$

\hspace*{\fill}$\Box$\\

From equation (\ref{3-2}), one can infer

\begin{eqnarray*}
z_t + z + (I-\partial_{xx})^{-1}(zz_x + (zv)_x) =-(I-\partial_{xx})^{-1}vv_x,
\end{eqnarray*} whence if applying $I_N$

\begin{eqnarray}\label{3-4}
I_Nz(t) = \int_0^t e^{-(t-s)} I_N(I-\partial_{xx})^{-1}(zz_x+(zv)_x + vv_x)(s)ds.   
\end{eqnarray}

Let $$S_v = \{I_Nz\in Y^1_{0, T}| \|I_Nz\|_{Y^1_{0, T}}\le \|I_Nv\|_{Y^1_{0, T}}\},$$

and mapping $G: q\mapsto G(q)$ defined as

$$G(q; v) = -\int_0^te^{-(t-s)}I_N(I-\partial_{xx})^{-1}(qq_x+(qv)_x + vv_x)ds, \,\,t\le T.$$

 We aim to use Fixed Point Theorem in $S_v$ to show the existence of solution $q=I_Nz\in S_v$ to  the integration equation (\ref{3-4}). Following the fashion in \cite{WZ, LW} along with bilinear estimate Lemma \ref{le3-3}, we can find that if $\|I_N v\|_{Y^1_{0, T}}$ is sufficiently small: $\|I_Nv\|_{Y^1_{0, T}}\lesssim {1\over C_T}$, then for $q, \tilde q\in S_v$, there hold

 \begin{eqnarray*}
&&\|G(q)\|_{Y^1_{0, T}}\le \|I_N v\|_{Y^1_{0,T}}, \\
&&\|G(q)-G(\tilde q)\|_{Y^1_{0, T}}\le {1\over 2}\|I_N q - I_N\tilde q\|_{Y^1_{0, T}},
 \end{eqnarray*} to conclude the fixed point argument confirming $I_Nz\in S_v$. If $\tau>0$, and we consider mapping $\tilde G$:
 
 $$\tilde G(q; v) = -\int_{\tau}^{t}e^{-(t-s)}I_N(I-\partial_{xx})^{-1}(qq_x+(qv)_x + vv_x)ds,$$

 let bound $M = 2(\|I_N(\tau)\|_{H^1}+\|I_Nv\|_{Y^1_{\tau, T}})$ and set $\tilde S_M = \{I_Nz\in Y^1_{0, T}| \|I_Nz\|_{Y^1_{0, T}}\le 2\|I_Nz(\tau)\|_{H^1} + \|I_Nv\|_{Y^1_{0, T}}\},$ similar argument via Fixed Point Theorem of $\tilde G$ can lead to existence of $I_Nz$ satisfying 

 $$\|I_Nz\|_{Y^1_{\tau, T}}\le 2(\|I_N z(\tau)\|_{H^1}+\|I_N v\|_{Y^1_{\tau, T}}).$$
 
Therefore, we obtain

\begin{Lemma}\label{le3-5}
Problem (\ref{3-2}) admits a solution $I_Nz\in Y^1_{0, T}$ satisfying

\begin{eqnarray}
\|I_Nz\|_{Y^1_{0, T}}\le C\|I_Nv\|_{Y^1_{0, T}}.
\end{eqnarray} 

If $\tau>0$, on $[\tau, \tau+T]$,
\begin{eqnarray}\label{3-7}
\|I_Nz\|_{Y^1_{\tau, T}}\le 2\|I_Nz\|_{H^1} + 2\|I_Nv\|_{Y^1_{\tau, T}}.
\end{eqnarray}
\end{Lemma}

\begin{Lemma}
If $\tau>0$, $\|I_Nv\|_{Y^1_{0,T}}\le \delta'$, then there holds

\begin{eqnarray}
\|I_Nz\|_{Y^1_{\tau, T}}\le \|I_N v\|_{Y^1_{0,T}}. 
\end{eqnarray}
\end{Lemma}
{\bf Proof.} From mild solution form defined on $[\tau, \tau+t]$, one writes

\begin{eqnarray*}
I_Nz(\tau+t)&=&e^{-t}I_Nz(\tau)-\int_{\tau}^{\tau+t}e^{-(\tau+t-s)}(I-\partial_{xx})^{-1}I_N(zz_x+zv_x+z_xv)ds\\
&&\quad-\int_{\tau}^{\tau+t}e^{-(\tau+t-s)}(I-\partial_{xx})^{-1}I_Nvv_xds\\ &=&W(t)z(\tau) + P(\tau+t, \tau, z(\tau), v)+Q(\tau+t, \tau, v)
\end{eqnarray*}

On mesh $\{T, 2T, 3T, \cdots, kT, \cdots\}$ with general $k\in \mathbb{N}$ and define $I_Nz_k = I_Nz(kT)$, we have

\begin{eqnarray*}
I_Nz_k = W(T)I_Nz_{k-1} + P(kT, (k-1)T, I_Nz_{k-1}, I_Nv) + Q(kT, (k-1)T, I_Nv).
\end{eqnarray*}

Applying Lemma \ref{le3-3} and Cauchy-Schwartz inequality, we infer that 

\begin{eqnarray*}
\|I_Nz_k\|_{H^1}\le(e^{-T}+C_T\|I_Nz_{k-1}\|)\|I_Nz_{k-1}\|_{H^1} + C_T\|I_Nv\|^2_{H^1}.
\end{eqnarray*} 

Owing to Lemma \ref{le3-5}, $I_Nz$ is sufficiently small in $Y^1_{\tau, T}$ as $\|I_Nv\|_{Y^1_{\tau, T}}<\delta'$ for any $\tau>0$. Denote $\zeta=e^{-T}+C_T\delta'$, then 

\begin{eqnarray*}
\|I_Nz_k\|_{H^1}&\le& \zeta\|I_nz_{k-1}\| + C_T\|I_Nv\|^2_{Y^1_{(k-1)T, T}}\\
&\le& \zeta^2\|I_Nz_{k-2}\|_{H^1}+C_T\|I_Nv\|^2_{Y^1_{(k-1)T, T}}\\
&\vdots&\\
&\le&\zeta^kC_T\|I_Nv\|^2_{Y^1_{(k-1)T, T}}
\end{eqnarray*} as $z_0\equiv 0$.\\

Relying on the above result of $I_Nz_k$ and estimate (\ref{3-7}), $\|I_N\|_{kT,T}$ is proved immediately. Any $\tau\neq kT$ also holds same estimate using the argument in Lemma 3.4 of \cite{WZ}, and we complete the proof. 
\hspace*{\fill}$\Box$\\

To approach the stability of desired solution, we need a decay property of error $w$ to the BBM equation. This attribute leads to asymptotic periodicity of any non-periodic solution. From dynamics of PDEs, this asymptotic time-periodicity is caused by force $f$ specified to be time-periodic. The error equation is proposed on $\mathbb{R}_+\times \mathbb{R}$ as

\begin{eqnarray}\label{3-8}
\begin{cases}
w_t - w_{xxt} +w - w_{xx} + (aw)_x = 0,\\
w(x, 0) = \psi(x).
\end{cases}
\end{eqnarray}

In it, the reference $a$ is a function of $(x,t)$, satisfying certain type of smallness. 

\begin{Lemma}\label{le3-7}
If $\|I_Na\|_{Y^1_{\tau, T}}$ is bounded for $\tau\in \mathbb{R}_+$, $N$ can be selected to be $N_0$ so large that a constant $\gamma_0 = 1 - CN_0^{-{3\over 2}+}\|I_Na\|_{H^1}>0$, then $\|w\|_{Y^\ell_{\tau, T}}\lesssim e^{-2\gamma_0\tau}$.
\end{Lemma}
{\bf Proof.} Applying $I_N$, we have 

\begin{eqnarray*}
\partial_t I_N w - \partial_t I_Nw_{xx} + I_Nw - I_Nw_{xx} + I_N(aw)_x = 0,
\end{eqnarray*}

whence $L^2$ energy identity gives $H^1$-estimate

\begin{eqnarray*}
{1\over 2}{d\over dt}\|I_Nw(t)\|^2_{H^1} + \|I_Nw(t)\|^2_{H^1} + (I_N(aw)_x, I_Nw) = 0. 
\end{eqnarray*}

From Lemma \ref{le3-4}, there holds a trilinear estimate $$|(I_N(aw)_x, I_Nw)|\lesssim N^{-{3\over 2}+}\|I_Nw\|^2_{H^1}\|I_Na\|_{H^1}.$$

Thus if $N = N_0$ being so large that $\gamma_0 = 1 - CN_0^{-{3\over 2}+}\|I_{N_0} a\|_{H^1}>0$, then exponential decay follows as

$$\|I_Nw(t)\|_{H^1}\le e^{-2\gamma_0 t}\|I_N\psi\|_{H^1}^2$$

for any $t>0$. On the other term of $\|I_Nw\|_{Y^1_{\tau, T}}$ with $\tau >0$, 

\begin{eqnarray*}
2\int_{\tau}^{\tau+T}\|I_Nw(s)\|_{H^1}^2\le \|I_Nw(\tau)\|^2_{H^1} + \|I_Nw(\tau+T)\|_{H^1}^2 + 2CN^{-{3\over 2}+}\int_{\tau}^{\tau+T}\|I_Na(s)\|_{H^1}\|I_Nw(s)\|^2_{H^1}ds.
\end{eqnarray*} Owing to $N=N_0$ and $\gamma_0 = 1 - CN_0^{-{3\over 2}+}\|I_{N_0}a\|_{H^1}>0$, integration term on the right side is absorbed to the left, therefore combining with estimate of $\|I_Nw(t)\|_{H^1}$ just obtained, we can infer the result owing to Lemma \ref{le3-1} with decay rate $2\gamma_0$. 
\hspace*{\fill}$\Box$\\

\section{Proof of main theorems}\label{sec4}
\setcounter{equation}{0}

\noindent {\bf Proof of Theorem \ref{th2-1}.} Given $\tau, T>0$, owing to Lemmas \ref{le3-2} and \ref{le3-5}, we obtain the sum $I_Nu\in Y^{1}_{\tau, T}$ thus $u\in Y^\ell_{\tau, T}$ by Lemma \ref{le3-1}. \hspace*{\fill}$\Box$\\

\noindent {\bf Proof of Theorem \ref{th2-2}.} First, we would prove the existence of periodic solution. Note that the local well-posedness has been proved by Theorem \ref{th2-1}. \\

In this part, we will prove asymptotic periodicity, i.e., $u$ generated by any $\phi$ is asymptotically time-periodic. Let $w = u(t+\theta, x)-u(t,x)$, then $w$ exactly follows (\ref{3-8}) with $\psi(x) = u(\theta, x)-\phi(x), a={1\over 2}(u(t,x)+u(t+\theta, x))$. If $\delta$ is sufficiently small, $\|I_Na\|_{H^1}$ possesses smallness such that $\gamma_0>0$. Lemma \ref{le3-7} explicitly show the asymptotic periodicity as $\tau\rightarrow \infty$:

$$\|u(\cdot+\theta)-u(\cdot)\|_{Y^{\ell}_{\tau, T}}\lesssim e^{-2\gamma_1\tau}.$$

If $u_k=u(k\theta, x), k\in \mathbb{N}$, one can show that sequence $\{u_k\}_k$ is of Cauchy in $H^{\ell}$. Passing to limit, $\displaystyle\lim_{k\rightarrow\infty}u_k = \tilde\phi$ which is unique and we extract in next step as the very initial data to generate precisely periodic solution $\tilde u\in H^{\ell}$. It is noteworthy to point out that  detailed fashion in \cite{BSZ2} by Bona et al and later followers \cite{UsmanZ1, UsmanZ2,WZ,GWX} to help approach this Cauchy sequence result. It can be seen from proof that size of $\tilde\phi$ must be small thus $\tilde u$ is small in norm. \\

There also holds uniqueness of $\tilde u$. If there are two distinct periodic solutions $\tilde u_1$ and $\tilde u_2$. Let $w=\tilde u_1(t,x) - \tilde u_2(t,x)$ which satisfies (\ref{3-8}) again. Since Lemma {\ref{le3-7}}, $w\equiv 0$ as initial data $\psi = \tilde\phi-\tilde\phi=0.$\\

Second, we address the local stability of the solution. Let $w = u(x, t)-\tilde u(x,t), a={1\over 2}(u(t,x)+\tilde u(t, x))$. The $w$ satisfies (\ref{3-8}). According to smallness condition posed in Lemma \ref{le3-2}, $I_Na$ is able to inherit smallness on $(\phi, f)$ from $u(t, x)$ and $\tilde u(t, x)$ such that $\gamma_1>0$. Therefore, exponential decay of $w$ from Lemma \ref{le3-7} leads to the desired stability.\\

It turns out that this local stability can be improved to global stability alike in dynamical behavior of solutions to KdV, Burgers, and 2D hydrodynamical models  (see e.g. \cite{BSZ3, UsmanZ1, WZ, GWX}). This is achieved by global absorbing property of $I_Nu$. In fact, since

\begin{eqnarray*}
\partial_t I_Nu - \partial_tI_Nu_{xx} + I_Nu  - I_Nu_{xx} + I_N(uu_x)= I_Nf,
\end{eqnarray*}

the $L^2$ energy estimate leads to 

\begin{eqnarray*}
{1\over 2}{d\over dt}\|I_Nu(t)\|_{H^1}^2 + \|I_Nu(t)\|_{H^1}^2 = (I_N(f-uu_x), I_Nu).
\end{eqnarray*}



Note that 

\begin{eqnarray*}
|(I_Nf, I_Nf)|&\le& {1\over 2}\|f\|^2+{1\over 2}\|I_Nu\|^2,\\
|(I_Nuu_x, I_Nu)|&\le& CN^{-{3\over 2}+}\|I_Nu\|^3_{H^1},
\end{eqnarray*}

the latter of which follows the same fashion in Lemma \ref{le3-7} or Lemma \ref{le3-4}. If we select $N_1$ sufficiently large, so that $\gamma_1 = {1} - CN_1^{-{3\over 2}+}\|I_Nu\|_{H^1}>0$, then we arrive at

\begin{eqnarray*}
{1\over 2}{d\over dt}\|I_Nu(t)\|_{H^1}^2+\gamma_1\|I_Nu(t)\|^2 \le {1\over 2}\|f(t)\|^2,
\end{eqnarray*}

whence the absorbing inequality holds as

\begin{eqnarray*}
\|I_Nu(t)\|^2\le e^{-{2\gamma_1 t}}\|I_N\phi\|^2_{H^1}+\int_0^te^{-2\gamma_1{(t-s)}}\|f(s)\|^2ds\le e^{-{2\gamma_1 t}}\|I_N\phi\|^2_{H^1}+{1\over 2\gamma_1}\sup_{t\in[0, \infty]}\|f(t)\|^2. 
\end{eqnarray*}

This absorbing indicates that, for any $I_N\phi\in H^1$, there is $\tau>0$ sufficiently large so that $I_Nu(t)$ will be absorbed to a small ball around $0$ in $H^1$ as $t>\tau_0$ when $\gamma_0>0$, the local stability takes over the behavior of the flow, which turns out to be a global one. 
\hspace*{\fill}$\Box$\\

\begin{Remark}
Similar work can be formulated to prove $H^{\ell}$-periodic solutions to BBM-Burgers equation posed on 1D torus $\mathbb{T}$:

\begin{eqnarray*}
u_t - u_{txx} - u_{xx} + u_x + uu_x = f.
\end{eqnarray*}

In this discussion, the large-time stability behavior is granted by exponential decays of semigroup generated by Burgers term $-u_{xx}$; and there is no need to add zero-order term. The large-time behavior -- existence of global attractors, was proved in \cite{MWang2}.  
\end{Remark}

\begin{Remark}
We also could continue to establish the result in energy spaces $H^\ell, \ell\in[1, 2].$ This result is same as in preprint \cite{LW} by same present authors. 
\end{Remark}

\begin{Remark}
This paper is devoted to show the work which pushes the existence results of solution to BBM equation on $\mathbb{R}$ down to low-regularity spaces (lower than classical energy space). As pointed in \cite{Tzvet}, I-energy method pursuing LWP only can work in sharply lower spaces $\ell = 0$ of $H^{\ell}$, but can not lower as $\ell<0$. On the counter part: damped KdV equations, the existence of global attractors have been found not only in current spaces but also even lower ones (see e.g. \cite{Goube1,Goube2, Tsugawa}).  
\end{Remark}

\section*{Acknowledgment}

\noindent Authors would use this opportunity to thank referees and editor's diligent work and precious comments. \\

\noindent One of the authors, Taige Wang, is constantly supported by Faculty Development Fund granted by College of Arts and Sciences, University of Cincinnati, and Taft Award by Taft Research Center, University of Cincinnati.

\end{document}